\PassOptionsToPackage{
    colorlinks,
    linkcolor={red!50!black},
    citecolor={blue!50!black},
    urlcolor={blue!80!black}}{hyperref}
\documentclass{article}
\usepackage{amsmath,amsfonts,amsthm,amssymb,mathtools,url}
\usepackage{natbib}
\usepackage[nottoc,notlof,notlot]{tocbibind}
\usepackage[utf8]{inputenc}
\usepackage[english]{babel}
\usepackage{algorithm2e}
\usepackage{pgfplots}
\usepackage{hyperref}
\usepackage{xcolor}
\usepackage{tikz}
\usepackage{subfig}
\usepackage{enumitem}

\newcommand{\R}{\mathbb{R}}

\newcommand{\C}{\mathbb{C}}

\newcommand{\N}{\mathbb{N}}

\newcommand{\scal}[2]{\left\langle #1,#2 \right\rangle}

\renewcommand{\Im}{\mathrm{Im}\,}
\renewcommand{\Re}{\mathrm{Re}\,}

\DeclareMathOperator{\Card}{Card}

\DeclareMathOperator{\phase}{phase}
\DeclareMathOperator{\Range}{Range}

\newtheorem{thm}{Theorem}[section]
\newtheorem{lem}[thm]{Lemma}
\newtheorem*{lem*}{Lemma}
\newtheorem{prop}[thm]{Proposition}
\newtheorem{cor}[thm]{Corollary}
\newtheorem{conj}[thm]{Conjecture}
\newtheorem*{conj*}{Conjecture}
\newtheorem*{thm*}{Theorem}
\newtheorem*{prop*}{Proposition}

\title{Phase retrieval with random Gaussian sensing vectors by alternating projections}
\author{Irène Waldspurger
  \thanks{MIT Institute for Data, Systems and Society; e-mail address: \texttt{waldspur@math.mit.edu}.}
  }
\date{}

\begin{document}
\maketitle

\begin{abstract}
We consider a phase retrieval problem, where we want to reconstruct a $n$-dimensional vector from its phaseless scalar products with $m$ sensing vectors, independently sampled from complex normal distributions. We show that, with a suitable initalization procedure, the classical algorithm of alternating projections succeeds with high probability when $m\geq Cn$, for some $C>0$. We conjecture that this result is still true when no special initialization procedure is used, and present numerical experiments that support this conjecture.
\end{abstract}

\section{Introduction}

A \emph{phase retrieval} problem consists in recovering an unknown vector $x_0\in\C^n$ from $m$ phaseless linear measurements, of the form
\begin{equation*}
b_k=|\scal{a_k}{x_0}|,\quad\quad k=1,\dots,m.
\end{equation*}
Such problems naturally appear in various applications, notably in optical imaging \citep{schechtman}. A lot of efforts have thus been made to design efficient algorithms that could numerically solve these problems.

The oldest reconstruction algorithms \citep{gerchberg, fienup} were iterative: they started from a random initial guess of $x_0$, and tried to iteratively refine it by various heuristics. These methods sometimes succeed, but can also fail to converge towards the correct solution: they may get stuck in \emph{stagnation points}, whose existence is due to the non-convexity of the problem. When they convergence and when they do not is not clearly understood.

To overcome convergence problems, convexification methods have been introduced \citep{chai,candes2}. Their principle is to lift the non-convex problem to a matricial space, where it can be approximated by a convex problem. These methods provably reconstruct the unknown vector $x_0$ with high probability if the sensing vectors $a_k$ are ``random enough'' \citep{candes_li,candes_li2,gross}. Numerical experiments show that they also perform well on more structured, non-random phase retrieval problems \citep{sun_smith,maxcut}.


Unfortunately, this good precision comes at a high computational cost: optimizing an $n\times n$ matrix is much slower than directly reconstructing a $n$-dimensional vector. More recently, a new family of algorithms has thus been developed, which enjoy similar theoretical guarantees as convexified methods, but have a much smaller computational complexity. The algorithms of this family rely on the following two-step scheme:
\begin{enumerate}[label=(\arabic*)]
\item an initialization step, that returns a point close to the solution;\label{item:step1}
\item a gradient descent (with possible additional refinements) over a well-chosen non-convex cost function.\label{item:step2}
\end{enumerate}
The intuitive reason why this scheme works is that the cost function, although globally non-convex, enjoys some good geometrical property in a neighborhood of the solution (like convexity or a weak form of it \citep{white}). When the initial point belongs to this neighborhood, the gradient descent then converges to the correct solution.

A preliminary form of this scheme appeared in \citep{netrapalli}, with an alternating minimization in step \ref{item:step2} instead of a gradient descent. \citep{candes_wirtinger} then proved, for a specific cost function, that this two-step scheme was able to exactly reconstruct the unknown $x_0$ with high probability, in the regime $m=O(n\log n)$, if the sensing vectors were independent and Gaussian. In \citep{candes_wirtinger2,kolte}, the same result was shown in the regime $m=O(n)$ for a slightly different cost function, with additional truncation steps. In \citep{zhang,wang}, it was extended to a different, non-smooth, cost function.

These new methods enjoy much stronger theoretical guarantees than ``traditional'' algorithms. However, it is not clear whether they really perform better in applications, or whether they actually behave similarly, and are simply easier to theoretically study. Traditional algorithms are well-known, simple and very easy to implement; understanding how they compare to more modern methods is of much value for applications.

In this article, we take a first step towards this goal, by considering the very classical alternating projections algorithm, introduced by \citep{gerchberg}, arguably the most simple and widely used method for phase retrieval. We show that, in the setting where sensing vectors are independent and Gaussian, it performs as well as gradient descent over a suitable cost function: it converges linearly to the true solution with high probability, provided that it is correctly initialized.

\begin{thm*}[See Corollary \ref{cor:global_convergence}]
There exist absolute constants $C_1,C_2,M>0$, $\delta\in]0;1[$ such that, if $m\geq Mn$ and the sensing vectors are independently chosen according to complex normal distributions, the sequence of iterates $(z_t)_{t\in \N}$ produced by the alternating projections method satisfies
\begin{equation*}
\forall t\in\N^*,\quad\quad \inf_{\phi\in\R}||e^{i\phi}x_0-z_t||\leq \delta^t||x_0||,
\end{equation*}
with probability at least
\begin{equation*}
1-C_1\exp(-C_2m),
\end{equation*}
provided that alternating projections are correctly initialized, for example with the method described in \citep{candes_wirtinger2}.
\end{thm*}
Several authors have already tried to establish properties of this kind, but, compared to ours, their results were significantly suboptimal in various respects. Using transversality arguments, \citep{noll,chen_fannjiang} have shown the local convergence of alternating projections for relatively general families of sensing vectors. Unfortunately, transversality arguments give no control on the convergence radius of the algorithm, which can be extremely small. Lower bounding the radius requires using the statistical properties of the sensing vectors. This was first attempted in \citep{netrapalli}. The result obtained by these authors was very similar to ours, but required resampling the sensing vectors at each step of the algorithm, an operation that is almost never done in practice. For a non resampled version, a preliminary result was given in \citep{soltanolkotabi}, but this result did not capture the correct scaling of the convergence radius as a function of $m$ and $n$. As a consequence, it only established global convergence to the correct solution for a suboptimal number of measurements ($m=O(n\log^2n)$), and with a complex initialization procedure.

To have theory and practice exactly coincide, the role of the initialization procedure should also be examined: in applications, alternating projections are often used with a random initialization, and not with a carefully chosen one.
Our numerical experiments indicate that removing the careful initialization does not significantly alter the convergence of the algorithm (still in the setting of Gaussian independent sensing vectors). This fact is related to the observations of \citep{sun_qu_wright}, who proved that, at least in the regime $m\geq O(n\log^3n)$ and for a specific cost function, the initialization part of the two-step scheme is not necessary in order for the algorithm to converge. In the context of alternating projections, we were however not able to prove a similar result\footnote{The proof technique used by \citep{sun_qu_wright} collapses when gradient descent is replaced with alternating projections, as described in an extended version of this article \citep{gerchberg_saxton_long}.}, and leave it as a conjecture.
\begin{conj*}[See Conjecture \ref{conj:convergence_random}]
Let any $\epsilon>0$ be fixed. When $m\geq Cn$, for $C>0$ large enough, alternating projections, starting from a random initialization (chosen according to a rotationally invariant distribution), converge to the true solution with probability at least $1-\epsilon$.
\end{conj*}

The article is organized as follows. Section \ref{s:setup} precisely defines phase retrieval problems and the alternating projections algorithm. Section \ref{s:with_init} states and proves the main result: the global convergence of alternating projections, with proper initialization, for $m=O(n)$ independent Gaussian measurements. Section \ref{s:numerical} contains numerical experiments, and presents our conjecture about the non-necessity of the initialization procedure.

This article only considers the most simple setting, where sensing vectors are independent and Gaussian, and measurements are not noisy. We made this choice in order to keep the technical content simple, but we hope that our results extend to more realistic settings, and future work should examine this issue.

\subsection{Notations}

For any $z\in\C$, $|z|$ is the modulus of $z$. We extend this notation to vectors: if $z\in\C^k$ for some $k\in\N^*$, then $|z|$ is the vector of $(\R^+)^k$ such that
\begin{equation*}
|z|_i=|z_i|,\quad\quad\forall i=1,\dots,k.
\end{equation*}
For any $z\in\C$, we set $E_{\phase}(z)$ to be the following subset of $\C$:
\begin{equation*}
\begin{array}{rll}
E_{\phase}(z)&=\left\{\frac{z}{|z|}\right\}&\mbox{ if }z\in\C-\{0\};\\
&=\{e^{i\phi},\phi\in\R\}&\mbox{ if }z=0.
\end{array}
\end{equation*}
We extend this definition to vectors $z\in\C^k$:
\begin{equation*}
E_{\phase}(z)=\prod_{i=1}^k E_{\phase}(z_i).
\end{equation*}
For any $z\in\C$, we define $\phase(z)$ by
\begin{equation*}
\begin{array}{rll}
\phase(z)&=\frac{z}{|z|}&\mbox{ if }z\in\C-\{0\};\\
&=1&\mbox{ if }z=0,
\end{array}
\end{equation*}
and extend this definition to vectors $z\in\C^k$, as for the modulus.

We denote by $\odot$ the pointwise product of vectors: for all $a,b\in\C^k$, $(a\odot b)$ is the vector of $\C^k$ such that
\begin{equation*}
(a\odot b)_i=a_ib_i,\quad\quad\forall i=1,\dots,k.
\end{equation*}
We define the operator norm of any matrix $A\in\C^{n_1\times n_2}$ by
\begin{equation*}
|||A||| = \sup_{v\in \C^{n_2}, ||v||=1}||Av||.
\end{equation*}
We denote by $A^\dag$ its Moore-Penrose pseudo-inverse. We note that $AA^\dag$ is the orthogonal projection onto $\Range(A)$. 

\section{Problem setup\label{s:setup}}

\subsection{Phase retrieval problem}

Les $n,m$ be positive integers. The goal of a phase retrieval problem is to reconstruct an unknown vector $x_0\in \C^n$ from $m$ measurements with a specific form.

We assume $a_1,\dots,a_m\in\C^n$ are given; they are called the \textit{sensing vectors}. We define a matrix $A\in\C^{m\times n}$ by
\begin{equation*}
A=\begin{pmatrix}a_1^*\\\vdots\\a_m^*\end{pmatrix}.
\end{equation*}
This matrix is called the \textit{measurement matrix}. The associated \textit{phase retrieval} problem is:
\begin{equation}\label{eq:problem_statement}
\mbox{reconstruct }x_0\mbox{ from }b\overset{def}{=}|Ax_0|.
\end{equation}
As the modulus is invariant to multiplication by unitary complex numbers, we can never hope to reconstruct $x_0$ better than \textit{up to multiplication by a global phase}. So, instead of exactly reconstructing $x_0$, we want to reconstruct $x_1$ such that
\begin{equation*}
x_1 = e^{i\phi}x_0,\quad\quad \mbox{for some }\phi\in\R.
\end{equation*}

In all this article, we assume the sensing vectors to be independent realizations of centered Gaussian variables with identity covariance:
\begin{equation}\label{eq:def_A}
(a_{i})_j\sim\mathcal{N}\left(0,\frac{1}{2}\right)
+\mathcal{N}\left(0,\frac{1}{2}\right)i,\quad\quad
\forall 1\leq i\leq m,1\leq j\leq n.
\end{equation}
The measurement matrix is in particular independent from $x_0$.

\citep{balan} and \citep{conca} have proved that, for \textit{generic} measurement matrices $A$, Problem \eqref{eq:problem_statement} always has a unique solution, up to a global phase, provided that $m\geq 4n-4$. In particular, with our measurement model \eqref{eq:def_A}, the reconstruction is guaranteed to be unique, with probability $1$, when $m\geq 4n-4$.

\subsection{Alternating projections\label{ss:alternating_projections}}

The alternating projections method has been introduced for phase retrieval problems by \citep{gerchberg}. It focuses on the reconstruction of $Ax_0$; if $A$ is injective, this then allows to recover $x_0$.

To reconstruct $Ax_0$, it is enough to find $z\in\C^m$ in the intersection of the following two sets.
\begin{enumerate}[label={(\arabic*)}]
\item $z\in \{z'\in\C^m,|z'|=b\}$;
\item $z\in\Range(A)$.
\end{enumerate}
Indeed, when the solution to Problem \eqref{eq:problem_statement} is unique, $Ax_0$ is the only element of $\C^m$ that simultaneously satisfies these two conditions (up to a global phase).

A natural heuristic to find such a $z$ is to pick any initial guess $z_0$, then to alternatively project it on the two constraint sets. In this context, we call \textit{projection} on a closed set $E\subset\C^m$ a function $P:\C^m\to E$ such that, for any $x\in\C^m$,
\begin{equation*}
||x-P(x)||=\inf_{e\in E}||x-e||.
\end{equation*}
The two sets defining constraints (1) and (2) admit projections with simple analytical expressions, which leads to the following formulas:
\begin{subequations}\label{eq:gs_image}
\begin{align}
y'_k&= b \odot \phase(y_k);& \mbox{(Projection onto set (1))}\\
y_{k+1}&= (AA^\dag) y'_k.& \mbox{(Projection onto set (2))}
\end{align}
\end{subequations}
If, for each $k$, we define $z_k$ as the vector such that $y_k=Az_k$ \footnote{which exists and is unique, because $y_k$ belongs to $\Range(A)$ and $A$ is injective with probability $1$}, an equivalent form of these equations is:
\begin{equation*}
z_{k+1} = A^\dag(b\odot\phase(Az_k)).
\end{equation*}

The hope is that the sequence $(y_k)_{k\in\N}$ converges towards $Ax_0$. Unfortunately, it can get stuck in \textit{stagnation points}. The following proposition (proven in Appendix \ref{s:stagnation_points}) characterizes these stagnation points.
\begin{prop}\label{prop:stagnation_points}
For any $y_0$, the sequence $(y_k)_{k\in\N}$ is bounded. Any accumulation point $y_\infty$ of $(y_k)_{k\in\N}$ satisfies the following property:
\begin{equation*}
\exists u\in E_{\phase}(y_\infty),\quad\quad
(AA^\dag)(b\odot u)=y_\infty.
\end{equation*}
In particular, if $y_\infty$ has no zero entry,
\begin{equation*}
(AA^\dag)(b\odot \phase(y_\infty))=y_\infty.
\end{equation*}
\end{prop}
Despite the relative simplicity of this characteristic property, it is extremely difficult to exactly compute the stagnation points, determine their attraction basin or avoid them when the algorithm happens to run into them.

\section{Alternating projections with good initialization\label{s:with_init}}

In this section, we prove our main result: for $m=O(n)$ Gaussian independent sensing vectors, the method of alternating projections converges to the correct solution with high probability, if it is carefully initialized.

\subsection{Local convergence of alternating projections}

We begin with a key result, that we will need to establish our statement. This result is a local contraction property of the alternating projections operator $x\to A^\dag(b\odot\phase(Ax))$.

\begin{thm}\label{thm:local_convergence}
There exist $\epsilon,C_1,C_2,M>0$, and $\delta\in]0;1[$ such that, if $m\geq Mn$, then, with probability at least
\begin{equation*}
1-C_1\exp(-C_2m),
\end{equation*}
the following property holds: for any $x\in\C^n$ such that
\begin{equation*}\label{eq:hyp_x}
\inf_{\phi\in\R}||e^{i\phi}x_0-x||\leq \epsilon ||x_0||,
\end{equation*}
we have
\begin{equation}\label{eq:progres_lineaire}
\inf_{\phi\in\R}||e^{i\phi}x_0-A^\dag(b\odot\phase(Ax))||\leq \delta \inf_{\phi\in\R}||e^{i\phi}x_0-x||.
\end{equation}
\end{thm}

\begin{proof}
For any $x\in\C^n$, we can write $Ax$ as
\begin{equation}\label{eq:Ax_orth}
Ax = \lambda_x (Ax_0) + \mu_x v^x,
\end{equation}
where $\lambda_x\in\C,\mu_x\in\R^+$, and $v^x\in\Range(A)$ is a unitary vector orthogonal to $Ax_0$.

As we will see, the norm of $|||A^\dag|||$ can be upper bounded by a number arbitrarily close to $1$, so to prove Inequality \eqref{eq:progres_lineaire}, it is enough to focus on the following quantity:
\begin{align*}
\inf_{\phi\in\R}&||e^{i\phi}A x_0-b\odot \phase(Ax)||\\
=&\inf_{\phi\in\R}||e^{i\phi}|A x_0|\odot \phase(Ax_0)-|Ax_0|\odot \phase(Ax)||.
\end{align*}
As $Ax$ is close to $Ax_0$ (up to a global phase), we can use a kind of mean value inequality, formalized by the following lemma (proven in Paragraph \ref{ss:diff_phase}). In the lemma, the term $\Im(z/z_0)$ is the derivative of the $\phase$ function (up to phase shift); the term $1_{|z|\geq |z_0|/6}$ can be thought of as a second order term.
\begin{lem}\label{lem:diff_phase}
For any $z_0,z\in\C$,
\begin{equation*}
|\phase(z_0+z)-\phase(z_0)| \leq 2. 1_{|z|\geq |z_0|/6} + \frac{6}{5}\left|\Im\left(\frac{z}{z_0}\right)\right|.
\end{equation*}
\end{lem}
From this lemma, for any $i=1,\dots,m$,
\begin{align*}
&\quad|\phase(\lambda_x)(Ax_0)_i-(b\odot\phase(Ax))_i|\\
&=\left|\phase(\lambda_x)(Ax_0)_i-|Ax_0|_i\phase(\lambda_x(Ax_0)_i+\mu_x (v^x)_i)\right|\\
&= |Ax_0|_i\left|\phase(Ax_0)_i-\phase\left((Ax_0)_i+\frac{\mu_x}{\lambda_x} (v^x)_i\right) \right|\\
&\leq 2.|Ax_0|_i1_{|\mu_x/\lambda_x||v^x_i|\geq |Ax_0|_i/6} + \frac{6}{5}|Ax_0|_i\left|\Im \left(\frac{\frac{\mu_x}{\lambda_x}v^x_i}{(Ax_0)_i}\right)\right|\\
&= 2.|Ax_0|_i1_{|\mu_x/\lambda_x||v^x_i|\geq |Ax_0|_i/6} + \frac{6}{5}\left|\Im \left(\frac{\frac{\mu_x}{\lambda_x}v^x_i}{\phase((Ax_0)_i)}\right)\right|.
\end{align*}
As a consequence,
\begin{align}
&\quad||\phase(\lambda_x)(Ax_0)-b\odot\phase(Ax)||\nonumber\\
&\leq \Bigg|\Bigg|
2.|Ax_0|\odot 1_{|\mu_x/\lambda_x||v^x|\geq |Ax_0|/6}\nonumber \\
&\hskip 1cm + \frac{6}{5}\left|\Im \left(\left(\frac{\mu_x}{\lambda_x}v^x\right)\odot\overline{\phase(Ax_0)}\right)\right|\,
\Bigg|\Bigg|\nonumber\\
&\leq 2\left|\left|
|Ax_0|\odot 1_{6|\mu_x/\lambda_x||v^x|\geq |Ax_0|}\right|\right| \nonumber\\
&\hskip 1cm+ \frac{6}{5}\left|\left|\Im \left(\left(\frac{\mu_x}{\lambda_x}v^x\right)\odot\overline{\phase(Ax_0)}\right)\right|
\right|.\label{eq:error_sum}
\end{align}
Two technical lemmas allow us to upper bound the terms of this sum. The first one is proved in Paragraph \ref{ss:first_term}, the second one in Paragraph \ref{ss:second_term}.
\begin{lem}\label{lem:first_term}
For any $\eta>0$, there exists $C_1,C_2,M,\gamma>0$ such that the inequality
\begin{equation*}
||\,|Ax_0|\odot 1_{|v|\geq |Ax_0|}||\leq \eta ||v||
\end{equation*}
holds for any $v\in\Range(A)$ such that $||v||<\gamma ||Ax_0||$, with probability at least
\begin{equation*}
1-C_1\exp(-C_2m),
\end{equation*}
when $m\geq Mn$.
\end{lem}
\begin{lem}\label{lem:second_term}
For $M,C_1>0$ large enough, and $C_2>0$ small enough, when $m\geq M n$, the property
\begin{equation*}
||\Im(v\odot\overline{\phase(Ax_0)})||\leq \frac{4}{5}||v||
\end{equation*}
holds for any $v\in\Range(A)\cap \{Ax_0\}^\perp$, with probability at least
\begin{equation*}
1-C_1\exp(-C_2 m).
\end{equation*}
\end{lem}

We also recall a classical result, that allows us to control the norms of $A$ and $A^\dag$.
\begin{prop}[\citep{davidson}, Thm II.13]\label{prop:davidson}
If $A$ is chosen according to Equation \eqref{eq:def_A}, then, for any $t$, with probability at least
\begin{equation*}
1-2\exp\left(-mt^2\right),
\end{equation*}
we have, for any $x\in\C^n$,
\begin{align*}
\sqrt{m}\left(1-\sqrt{\frac{n}{m}}-t\right)||x||
&\leq ||Ax||;\\
||Ax||&\leq \sqrt{m}\left(1+\sqrt{\frac{n}{m}}+t\right)||x||.
\end{align*}
\end{prop}
This lemma says that, if $m\geq Mn$ for $M$ large enough, then $||Ax||/({\sqrt{m}||x||})$ is arbitrarily close to $1$, uniformly over $x\in\C^n-\{0\}$. In particular, $|||A|||/\sqrt{m}$ and $ |||A^\dag|||.\sqrt{m}$ can be upper bounded by constants arbitrarily close to $1$, with probability $1-2\exp(-C_2 m)$.

Let us set $\epsilon^x=\inf_{\phi\in\R}\frac{||e^{i\phi}x_0-x||}{||x_0||}\leq\epsilon$. We have
\begin{align*}
\epsilon^x |||A|||\,||x_0||
&\geq \inf_{\phi\in\R}||e^{i\phi}Ax_0-Ax||\\
&=\sqrt{(1-|\lambda_x|)^2||Ax_0||^2+|\mu_x|^2}.
\end{align*}
Bounding $|||A|||$ as in the last remark, we deduce from this inequality that, if $\epsilon$ is small enough,
\begin{equation}\label{eq:mu_over_lambda}
\frac{|\mu_x|}{|\lambda_x|}\leq 1.01\sqrt{m}\epsilon^x ||x_0||.
\end{equation}

Let $\eta>0$ be such that
\begin{equation*}
12\eta + \frac{24}{25}<0.98.
\end{equation*}
We define $\gamma>0$ as in Lemma \ref{lem:first_term}. Using Lemmas \ref{lem:first_term} and \ref{lem:second_term}, we can upper bound Equation \eqref{eq:error_sum} by
\begin{align*}
||\phase(\lambda_x)(Ax_0)&-b\odot \phase(Ax)||\\
&\leq \left(12\eta +\frac{24}{25}\right)\left|\frac{\mu_x}{\lambda_x}\right|\\
&\leq 0.98\times 1.01\sqrt{m} \epsilon^x||x_0||\\
&\leq 0.99 \sqrt{m}\epsilon^x ||x_0||.
\end{align*}
This holds for all $x$ such that $\inf_{\phi\in\R}||e^{i\phi}x_0-x|| \leq\epsilon ||x_0||$, with probability at least
\begin{equation*}
1-C_1'\exp(-C_2'm).
\end{equation*}
(From Equation \eqref{eq:mu_over_lambda} and Proposition \ref{prop:davidson}, the condition $||v||<\gamma ||Ax_0||$ in Lemma \ref{lem:first_term} is satisfied if $\epsilon>0$ is small enough.)

This implies
\begin{align*}
&\quad\inf_{\phi\in\R}||e^{i\phi}x_0-A^\dag(b\odot\phase(Ax))||\\
&\leq ||\phase(\lambda_x) x_0-A^\dag(b\odot\phase(Ax))||\\
&\leq |||A^\dag|||\,||\phase(\lambda_x)(Ax_0)-b\odot \phase(Ax)||\\
&\leq 0.99 |||A^\dag|||\sqrt{m}\epsilon^x ||x_0||\\
&= 0.99 |||A^\dag|||\sqrt{m}\inf_{\phi\in\R}||e^{i\phi}x_0-x|| .
\end{align*}
As we have seen, $|||A^\dag|||\sqrt{m}$ can be bounded by any constant larger than $1$ if $M$ is large enough (with high probability), so if $\delta\in]0;1[$ is close enough to $1$,
\begin{align*}
\inf_{\phi\in\R}||e^{i\phi}x_0-A^\dag(b\odot\phase(Ax))||
\leq \delta \inf_{\phi\in\R}||e^{i\phi}x_0-x||,
\end{align*}
and this holds, with high probability, for any $x$ such that $\inf_{\phi\in\R}||e^{i\phi}x_0-x||\leq\epsilon||x_0||$.

\end{proof}

\subsection{Global convergence}

In the last paragraph, we have seen that the alternating projections operator is contractive, with high probability, in an $\epsilon ||x_0||$-neighborhood of the solution $x_0$. This implies that, if the starting point of alternating projections is at distance at most $\epsilon||x_0||$ from $x_0$, alternating projections converge to $x_0$. So if we have a way to find such an initial point, we obtain a globally convergent algorithm.

Several initialization methods have been proposed that achieve the precision we need with an optimal number of measurements, that is $m=O(n)$. Let us mention the truncated spectral initialization by \citep{candes_wirtinger2} (improving upon the slightly suboptimal spectral initializations introduced by \citep{netrapalli} and \citep{candes_wirtinger}), the null initialization by \citep{chen_fannjiang} and the method described by \citep{gao_xu}. All these methods consist in computing the largest or smallest eigenvector of
\begin{equation*}
\sum_{i=1}^m\alpha_i a_ia_i^*,
\end{equation*}
where the $\alpha_1,\dots,\alpha_m$ are carefully chosen coefficients, that depend only on $b$.

The method of \citep{candes_wirtinger2}, for example, has the following guarantees.
\begin{thm}[Proposition 3 of \citep{candes_wirtinger2}]\label{thm:guarantee_init}
Let $\epsilon>0$ be fixed.

We define $z$ as the main eigenvector of
\begin{equation}
\frac{1}{m}\sum_{i=1}^m|a_i^*x_0|^2 a_ia_i^* 1_{|a_i^*x_0|^2\leq \frac{9}{m}\sum_{j=1}^m|a^*_ix_0|^2}.\label{eq:init_matrix}
\end{equation}
There exist $C_1,C_2,M>0$ such that, with probability at least
\begin{equation*}
1-C_1\exp(-C_2 m),
\end{equation*}
the vector $z$ obeys
\begin{equation*}
\inf_{\phi\in\R,\lambda\in\R^*_+}||e^{i\phi}x_0-\lambda z||\leq \epsilon ||x_0||,
\end{equation*}
provided that $m\geq M n$.
\end{thm}
Combining this initialization procedure with alternating projections, we get Algorithm \ref{alg:algo_complet}. As shown by the following corollary, it converges towards the correct solution, at a linear rate, with high probability, for $m=O(n)$.
\begin{algorithm}
\SetKwInOut{Input}{Input}
\SetKwInOut{Output}{Output}
\Input{$A\in\C^{m\times n},b=|Ax_0|\in\R^m,T\in\N^*$.}
\BlankLine
\textbf{Initialization:} set $z_0$ to be the main eigenvector of the matrix in Equation \eqref{eq:init_matrix}.

\For{$t=1$ \KwTo $T$}{
Set $z_{t}\leftarrow A^\dag(b\odot\phase(Az_{t-1}))$.}
\BlankLine
\Output{$z_T$.}
\BlankLine
\caption{Alternating projections with truncated spectral initialization\label{alg:algo_complet}}
\end{algorithm}
\begin{cor}\label{cor:global_convergence}
There exist $C_1,C_2,M>0,\delta\in]0;1[$ such that, with probability at least
\begin{equation*}
1-C_1\exp(-C_2m),
\end{equation*}
Algorithm \ref{alg:algo_complet} satisfies
\begin{equation}\label{eq:global_convergence}
\forall t\in\N^*,\quad\quad
\inf_{\phi\in\R}||e^{i\phi}x_0-z_t||\leq \delta^t ||x_0||,
\end{equation}
provided that $m\geq Mn$.
\end{cor}
\begin{proof}
Let us fix $\epsilon,\delta\in]0;1[$ as in Theorem \ref{thm:local_convergence}. Let us assume that the properties described in Theorems \ref{thm:local_convergence} and \ref{thm:guarantee_init} hold; it happens on an event of probability at least
\begin{equation*}
1-C_1\exp(-C_2m),
\end{equation*}
provided that $m\geq Mn$, for some constants $C_1,C_2,M>0$.

Let us prove that, on this event, Equation \eqref{eq:global_convergence} also holds.

We proceed by recursion. From Theorem \ref{thm:guarantee_init}, there exist $\phi\in\R,\lambda\in\R^*_+$ such that
\begin{equation*}
||e^{i\phi}x_0-\lambda z_0||\leq\epsilon||x_0||.
\end{equation*}
So, from Theorem \ref{thm:local_convergence}, applied to $x=\lambda z_0$,
\begin{align*}
\inf_{\phi\in\R}||e^{i\phi}x_0-z_1||&
=\inf_{\phi\in\R}||e^{i\phi}x_0-A^\dag(b\odot \phase(z_0))||\\
&=\inf_{\phi\in\R}||e^{i\phi}x_0-A^\dag(b\odot \phase(\lambda z_0))||\\
&\leq \delta\inf_{\phi\in\R} ||e^{i\phi}x_0-\lambda z_0||\\
&\leq \epsilon \delta ||x_0||.
\end{align*}
This proves Equation \eqref{eq:global_convergence} for $t=1$.

The same reasoning can be reapplied to also prove the equation for $t=2,3,\dots$.
\end{proof}

\subsection{Complexity\label{ss:complexity}}

Let $\eta>0$ be the relative precision that we want to achieve:
\begin{equation*}
\inf_{\phi\in\R}||e^{i\phi}x_0-z_T||\leq \eta||x_0||.
\end{equation*}
Let us compute the number of operations that Algorithm \ref{alg:algo_complet} requires to reach this precision.

 The main eigenvector of the matrix defined in Equation \eqref{eq:init_matrix} can be computed - up to precision $\eta$ - in approximately $O(\log(1/\eta)+\log(n))$ power iterations. Each power iteration is essentially a matrix-vector multiplication, and thus requires $O(mn)$ operations.\footnote{These matrix-vector multiplications can be computed without forming the whole matrix (which would require $O(mn^2)$ operations), because this matrix factorizes as
\begin{equation*}
\frac{1}{m}A^* \mathrm{Diag}(|Ax_0|^2\odot I) A,
\end{equation*}
where $I\in\R^m$ is such that $\forall i\leq m,I_i=1_{|A_ix_0|^2\leq\frac{9}{m}\sum_{j=1}^m|A_ix_0|^2}$.}
As a consequence, the complexity of the initialization is
\begin{equation*}
O(mn\left(\log(1/\eta)+\log(n)\right)).
\end{equation*}
Then, at each step of the \texttt{for} loop, the most costly operation is the multiplication by $A^\dag$. When performed with the conjugate gradient method, it requires $O(mn\log(1/\eta))$ operations. To reach a precision equal to $\eta$, we need to perform $O(\log(1/\eta))$ iterations of the loop. So the total complexity of Algorithm \ref{alg:algo_complet} is
\begin{equation*}
O(mn\left(\log^2(1/\eta)+\log(n)\right)).
\end{equation*}

Let us mention that, when $A$ has a special structure, there may exist fast algorithms for the multiplication by $A$ and the orthogonal projection onto $\Range(A)$. In the case of masked Fourier measurements considered in \citep{candes_li2}, for example, assuming that our convergence theorem still holds, despite the non-Gaussianity of the measurements, the complexity of each of these operations reduces to $O(m\log n)$, yielding a global complexity of
\begin{equation*}
O(m\log(n)(\log(1/\eta)+\log(n))).
\end{equation*}
The complexity is then almost linear in the number of measurements.

\begin{figure*}
\centering
\begin{tabular}{|c|c|c|}
\hline
&Alternating projections&Truncated Wirtinger flow\\\hline
Unstructured case&$O\left(mn\left(\log^2(1/\eta)+\log(n)\right)\right)$
&$O\left(mn\left(\log(1/\eta)+\log(n)\right)\right)$\\\hline
Fourier masks&$O\left(m\log(n)\left(\log(1/\eta)+\log(n)\right)\right)$
&$O\left(m\log(n)\left(\log(1/\eta)+\log(n)\right)\right)$\\\hline
\end{tabular}
\caption{Complexity of alternating projections with initialization, and truncated Wirtinger flow.\label{fig:complexity}}
\end{figure*}

As a comparison, Truncated Wirtinger flow, which is currently the most efficient known method for phase retrieval from Gaussian measurements, has an identical complexity, up to a $\log(1/\eta)$ factor in the unstructured case (see Figure \ref{fig:complexity}).

\section{Numerical experiments\label{s:numerical}}

In this section, we numerically validate Corollary \ref{cor:global_convergence}. We formulate a conjecture about the convergence of alternating projections with random initialization, in the regime $m=O(n)$.

The code used to generate Figures \ref{fig:with_init} and \ref{fig:without_init} is available at
\begin{center}
\url{http://www-math.mit.edu/~waldspur/code/alternating_projections_code.zip}.
\end{center}

\subsection{Alternating projections with careful initialization}

Corollary \ref{cor:global_convergence} states that alternating projections succeed with high probability, when they start from a good initial point, in the regime where the number of measurements is linear in the problem dimension ($m=O(n)$).

We use the initialization method described in \citep{candes_wirtinger2}, as presented in Algorithm \ref{alg:algo_complet}. We run the algorithm for various choices of $n$ and $m$, $3000$ times for each choice. This allows us to compute an empirical probability of success, for each value of $(n,m)$.

The results are presented in Figure \ref{fig:with_init}. They confirm that, when $m=Cn$, for a sufficiently large constant $C>0$, the success probability can be arbitrarily close to $1$.

\begin{figure}
  \centering
  \newlength\figureheight 
  \newlength\figurewidth 
  \setlength\figureheight{6cm} 
  \setlength\figurewidth{8cm}
%
%
\begin{tikzpicture}

\begin{axis}[%
width=0.662\figurewidth,
height=\figureheight,
at={(0\figurewidth,0\figureheight)},
scale only axis,
point meta min=0,
point meta max=1,
axis on top,
separate axis lines,
every outer x axis line/.append style={black},
every x tick label/.append style={font=\color{black}},
xmin=0.5,
xmax=15.5,
xtick={1,3,5,7,9,11,13,15},
xticklabels={{2},{6},{10},{14},{18},{22},{26},{30}},
xlabel={Signal size $n$},
every outer y axis line/.append style={black},
every y tick label/.append style={font=\color{black}},
y dir=reverse,
ymin=0.5,
ymax=17.5,
ytick={2,4,6,8,10,12,14,16},
yticklabels={{92},{80},{68},{56},{44},{32},{20},{8}},
ylabel={Number of measurements $m$},
axis background/.style={fill=white}
]
\addplot [forget plot] graphics [xmin=0.5,xmax=15.5,ymin=0.5,ymax=17.5] {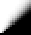};
\addplot [color=red,solid,line width=2.0pt,forget plot]
  table[row sep=crcr]{%
0.5	16.8333333333333\\
15.5	1.83333333333333\\
};
\end{axis}
\end{tikzpicture}%
  \caption{Probability of success for Algorithm \ref{alg:algo_complet}, as a function of $n$ and $m$. Black points indicate a probability equal to $0$, and white points a probability equal to $1$. The red line serves as a reference: it is the line $m=3n$.
    \label{fig:with_init}}
\end{figure}

\subsection{Alternating projections without careful initialization}

In a second experiment, we measure the probability that alternating projections succeed, when started from a random initial point (sampled from the unit sphere with uniform probability).

As previously, we compute this probability, for any pair $(m,n)$, by averaging the results of $3000$ tests. The results are presented in Figure \ref{fig:without_init}. From this figure, it seems that alternating projections behave similarly with and without a careful initialization procedure: they converge towards the correct solution as soon as $m\geq Cn$, for a suitable constant $C>0$. Only the precise value of $C$ depends on whether the initialization procedure is careful or not.

As a consequence, we have the following conjecture.
\begin{conj}\label{conj:convergence_random}
Let any $\epsilon>0$ be fixed. When $m\geq Cn$, for $C>0$ large enough, alternating projections with a random rotationally invariant initialization succeed with probability at least $1-\epsilon$.
\end{conj}

Complementary experiments (not shown here), suggest that, in the regime $m=O(n)$, there are attractive stagnation points, so there are initializations for which alternating projections fail. However, it seems that these bad initializations occupy a very small volume in the space of all possible initial points. Therefore, a random initialization leads to success with high probability.
Unfortunately, proving the conjecture a priori requires to evaluate in some way the size of the attraction basin of stagnation points, which seems difficult.

\begin{figure}
	\centering
	\setlength\figureheight{6cm} 
	\setlength\figurewidth{8cm}
%
%
\begin{tikzpicture}

\begin{axis}[%
width=0.662\figurewidth,
height=\figureheight,
at={(0\figurewidth,0\figureheight)},
scale only axis,
point meta min=0,
point meta max=1,
axis on top,
separate axis lines,
every outer x axis line/.append style={black},
every x tick label/.append style={font=\color{black}},
xmin=0.5,
xmax=15.5,
xtick={1,3,5,7,9,11,13,15},
xticklabels={{2},{6},{10},{14},{18},{22},{26},{30}},
xlabel={Signal size $n$},
every outer y axis line/.append style={black},
every y tick label/.append style={font=\color{black}},
y dir=reverse,
ymin=0.5,
ymax=17.5,
ytick={2,4,6,8,10,12,14,16},
yticklabels={{152},{132},{112},{92},{72},{52},{32},{12}},
ylabel={Number of measurements $m$},
axis background/.style={fill=white}
]
\addplot [forget plot] graphics [xmin=0.5,xmax=15.5,ymin=0.5,ymax=17.5] {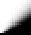};
\addplot [color=red,solid,line width=2.0pt,forget plot]
  table[row sep=crcr]{%
0.5	16.9\\
15.5	7.9\\
};
\end{axis}
\end{tikzpicture}%
	\caption{Probability of success for alternating projections with a random Gaussian initialization, as a function of $n$ and $m$. Black points indicate a probability equal to $0$, and white points a probability equal to $1$. The red line serves as a reference: it is the line $m=3n$.\label{fig:without_init}}
\end{figure}

\section*{Acknowledgments}

Part of this work has been done while the author was at École Normale Supérieure de Paris, where she has been partially supported by ERC grant InvariantClass 320959.

\appendix

\section{Proof of Proposition \ref{prop:stagnation_points}\label{s:stagnation_points}}

\begin{prop*}[Proposition \ref{prop:stagnation_points}]
For any $y_0$, the sequence $(y_k)_{k\in\N}$ is bounded. Any accumulation point $y_\infty$ of $(y_k)_{k\in\N}$ satisfies the following property:
\begin{equation*}
\exists u\in E_{\phase}(y_\infty),\quad\quad
(AA^\dag)(b\odot u)=y_\infty.
\end{equation*}
In particular, if $y_\infty$ has no zero entry,
\begin{equation*}
(AA^\dag)(b\odot \phase(y_\infty))=y_\infty.
\end{equation*}
\end{prop*}

\begin{proof}[Proof of Proposition \ref{prop:stagnation_points}]
The boundedness of $(y_k)_{k\in\N}$ is a consequence of the fact that $||y'_k||=||b||$ for all $k$, so $||y_{k+1}||\leq |||AA^\dag|||\,||b||$.

Let us show the second part of the statement. Let $y_\infty$ be an accumulation point of $(y_k)_{k\in\N}$, and let $\phi:\N\to\N$ be an extraction such that
\begin{equation*}
y_{\phi(n)}\to y_\infty\quad\mbox{when}\quad n\to+\infty.
\end{equation*}
By compacity, as $(y'_{\phi(n)})_{n\in\N}$ and $(y_{\phi(n)+1})_{n\in\N}$ are bounded sequences, we can assume, even if we have to consider replace $\phi$ by a subextraction, that they also converge. We denote by $y'_\infty$ and $y_\infty^{+1}$ their limits:
\begin{equation*}
y'_{\phi(n)}\to y'_{\infty}\quad\mbox{and}\quad y_{\phi(n)+1}\to y_\infty^{+1}\quad\mbox{when }n\to+\infty.
\end{equation*}

Let us define
\begin{equation*}
E_b=\{y'\in\C^m,|y'|=b\}.
\end{equation*}
We observe that, for any $k$,
\begin{equation*}
d(y'_{k-1},\Range(A))\geq d(y_{k},E_b)\geq d(y'_k,\Range(A)).
\end{equation*}
Indeed, because the operators $y\to b\odot\phase(y)$ and $y\to(AA^\dag)y$ are projections,
\begin{equation*}
\begin{array}{rcl}
d(y'_{k-1},\Range(A))=&d(y'_{k-1},y_k)&\geq d(y_k,E_b);\\
d(y_k,E_b)=&d(y_k,y'_k)&\geq d(y'_k,\Range(A)).
\end{array}
\end{equation*}
So the sequences $(d(y_k,E_b))_{k\in\N}$ and $(d(y'_k,\Range(A)))_{k\in\N}$ converge to the same non-negative limit, that we denote by $\delta$. In particular,
\begin{align*}
d(y_\infty,E_b)=\delta=d(y'_\infty,\Range(A)).
\end{align*}
If we pass to the limit the equalities $d(y_{\phi(n)},E_b)=||y_{\phi(n)}-y'_{\phi(n)}||$ and $d(y'_{\phi(n)},\Range(A))=||y'_{\phi(n)}-y_{\phi(n)+1}||$, we get
\begin{equation*}
||y_\infty-y'_\infty||=||y'_\infty-y_\infty^{+1}||=\delta=d(y'_\infty,\Range(A)).
\end{equation*}
As $\Range(A)$ is convex, the projection of $y'_\infty$ onto it is uniquely defined. This implies
\begin{equation*}
y_\infty=y_\infty^{+1},
\end{equation*}
and, because $\forall n,y_{\phi(n)+1}=(AA^\dag) y'_{\phi(n)}$,
\begin{equation*}
y_\infty=y_\infty^{+1}=(AA^\dag)y'_\infty.
\end{equation*}
To conclude, we now have to show that $y'_\infty=b\odot u$ for some $u\in E_{\phase}(y_\infty)$. We use the fact that, for all $n$, $y'_{\phi(n)}=b\odot \phase(y_{\phi(n)})$.

For any $i\in\{1,\dots,m\}$, if $(y_\infty)_i\ne 0$, $\phase$ is continuous around $(y_\infty)_i$, so $(y'_{\infty})_i= b_i\phase((y_\infty)_i)$. We then set $u_i=\phase((y_\infty)_i)$, and we have $(y'_\infty)_i=b_iu_i$.

If $(y_\infty)_i=0$, we set $u_i=\phase((y'_\infty)_i)\in E_{\phase}(0)=E_{\phase}((y_\infty)_i)$. We then have $y'_\infty=|y'_\infty|u_i=b_iu_i$.

With this definition of $u$, we have, as claimed, $y'_\infty=b\odot u$ and $u\in E_{\phase}(y_\infty)$.

\end{proof}

\section{Technical lemmas for Section \ref{s:with_init}}

\subsection{Proof of Lemma \ref{lem:diff_phase}\label{ss:diff_phase}}

\begin{lem*}[Lemma \ref{lem:diff_phase}]
For any $z_0,z\in\C$,
\begin{equation*}
|\phase(z_0+z)-\phase(z_0)| \leq 2. 1_{|z|\geq |z_0|/6} + \frac{6}{5}\left|\Im\left(\frac{z}{z_0}\right)\right|.
\end{equation*}
\end{lem*}
\begin{proof}
The inequality holds if $z_0=0$, so we can assume $z_0\ne 0$. We remark that, in this case,
\begin{equation*}
|\phase(z_0+z)-\phase(z_0)| = |\phase(1+z/z_0)-1|.
\end{equation*}
It is thus enough to prove the lemma for $z_0=1$, so we make this assumption.

When $|z|\geq 1/6$, the inequality is valid. Let us now assume that $|z|<1/6$. Let $\theta\in\left]-\frac{\pi}{2};\frac{\pi}{2}\right[$ be such that
\begin{equation*}
e^{i\theta}=\phase(1+z).
\end{equation*}
Then
\begin{align*}
|\phase(1+z)-1| &=|e^{i\theta}-1|\\
&=2|\sin(\theta/2)|\\
&\leq |\tan\theta|\\
&=\frac{|\Im(1+z)|}{|\Re(1+z)|}\\
&\leq \frac{|\Im(z)|}{1-|z|}\\
&\leq \frac{6}{5}|\Im(z)|.
\end{align*}
So the inequality is also valid.
\end{proof}

\subsection{Proof of Lemma \ref{lem:first_term}\label{ss:first_term}}

\begin{lem*}[Lemma \ref{lem:first_term}]
For any $\eta>0$, there exists $C_1,C_2,M,\gamma>0$ such that the inequality
\begin{equation*}
||\,|Ax_0|\odot 1_{|v|\geq |Ax_0|}||\leq \eta ||v||
\end{equation*}
holds for any $v\in\Range(A)$ such that $||v||<\gamma ||Ax_0||$, with probability at least
\begin{equation*}
1-C_1\exp(-C_2m),
\end{equation*}
when $m\geq Mn$.
\end{lem*}

\begin{proof}
For any $S\subset\{1,\dots,m\}$, we denote by $1_S$ the vector of $\C^m$ such that
\begin{align*}
(1_S)_j&=1\mbox{ if }j\in S\\
&=0\mbox{ if }j\notin S.
\end{align*}
We use the following two lemmas, proven in Paragraphs \ref{sss:S_geq_bm} and \ref{sss:S_leq_bm}.

\begin{lem}\label{lem:S_geq_bm}
Let $\beta\in]0;1/2[$ be fixed. There exist $C_1>0$ such that, with probability at least
\begin{equation*}
1-C_1\exp(-\beta^3m/e),
\end{equation*}
the following property holds: for any $S\subset\{1,\dots,m\}$ such that $\Card(S)\geq\beta m$,
\begin{equation}\label{eq:Ax01S}
||\,|Ax_0|\odot 1_S|| \geq \beta^{3/2}e^{-1/2}||Ax_0||.
\end{equation}
\end{lem}

\begin{lem}\label{lem:S_leq_bm}
Let $\beta\in\left]0;\frac{1}{100}\right]$ be fixed. There exist $M,C_1,C_2>0$ such that, if $m\geq M n$, then, with probability at least
\begin{equation*}
1-C_1\exp(-C_2 m),
\end{equation*}
the following property holds: for any $S\subset\{1,\dots,m\}$ such that $\Card (S)<\beta m$ and for any $y\in\Range(A)$,
\begin{equation}\label{eq:y1S}
||y\odot 1_S||\leq 10\sqrt{\beta\log(1/\beta)}||y||.
\end{equation}
\end{lem}

Let $\beta>0$ be such that $10\sqrt{\beta\log(1/\beta)}\leq \eta$. Let $M$ be as in Lemma \ref{lem:S_leq_bm}. We set
\begin{equation*}
\gamma = \beta^{3/2}e^{-1/2}.
\end{equation*}
We assume that Equations \eqref{eq:Ax01S} and \eqref{eq:y1S} hold; from the lemmas, this occurs with probability at least
\begin{equation*}
1-C_1'\exp(-C_2'm),
\end{equation*}
for some constants $C_1',C_2'>0$, provided that $m\geq Mn$.

On this event, for any $v\in\Range(A)$ such that $||v||<\gamma ||Ax_0||$, if we set $S_v=\{i\mbox{ s.t. }|v_i|\geq|Ax_0|_i\}$, we have that
\begin{equation*}
\Card S_v < \beta m.
\end{equation*}
Indeed, if it was not the case, we would have, by Equation \eqref{eq:Ax01S},
\begin{align*}
||v||&\geq ||v\odot 1_{S_v}||\\
&\geq ||\,|Ax_0|\odot 1_{S_v}||\\
&\geq \beta^{3/2}e^{-1/2}||Ax_0||\\
&=\gamma ||Ax_0||,
\end{align*}
which is in contradiction with the way we have chosen $v$.

So we can apply Equation \eqref{eq:y1S}, and we get
\begin{align*}
||\,|Ax_0|\odot 1_{|v|\geq |Ax_0|}||
&\leq ||v\odot 1_S||\\
&\leq 10\sqrt{\beta\log(1/\beta)}||v||\\
&\leq \eta ||v||.
\end{align*}

\end{proof}

\subsubsection{Proof of Lemma \ref{lem:S_geq_bm}\label{sss:S_geq_bm}}

\begin{proof}[Proof of Lemma \ref{lem:S_geq_bm}]
If we choose $C_1$ large enough, it is enough to show the property for $m$ larger than some fixed constant.

We first assume $S$ fixed, with cardinality $\Card S\geq\beta m$. We use the following lemma.
\begin{lem}[\citep{dasgupta}, Lemma 2.2]\label{lem:dasgupta}
Let $k_1<k_2$ be natural numbers. Let $X\in\C^{k_2}$ be a random vector whose coordinates are independent, Gaussian, of variance $1$. Let $Y$ be the projection of $X$ onto its $k_1$ first coordinates. Then, for any $t>0$,
\begin{align*}
\mbox{\rm Proba}\left(\frac{||Y||}{||X||}\leq \sqrt{\frac{t k_1}{k_2}}\right)
&\leq \exp\left(k_1(1-t+\log t)\right)&\mbox{if }t<1;\\
\mbox{\rm Proba}\left(\frac{||Y||}{||X||}\geq \sqrt{\frac{t k_1}{k_2}}\right)
&\leq \exp\left(k_1(1-t+\log t)\right)&\mbox{if }t>1.
\end{align*}
\end{lem}
From this lemma, for any $t\in]0;1[$, because $Ax_0$ has independent Gaussian coordinates,
\begin{align*}
P\left(\frac{||\,|Ax_0|\odot 1_S||}{||Ax_0||}\leq\sqrt{t\beta} \right)
\leq \exp\left(-\beta m(t-1-\ln t)\right).
\end{align*}
In particular, for $t=\frac{\beta^2}{e}$,
\begin{align}
&P\left(\frac{||\,|Ax_0|\odot 1_S||}{||Ax_0||}\leq \beta^{3/2}e^{-1/2} \right)\nonumber\\
&\hskip 1.5cm \leq \exp\left(-\beta m\left(\frac{\beta^2}{e}-2\ln \beta\right)\right).
\label{eq:P_Ax01S}
\end{align}
As soon as $m$ is large enough, the number of subsets $S$ of $\{1,\dots,m\}$ with cardinality $\lceil \beta m\rceil$ satisfies
\begin{align}
\binom{m}{\lceil\beta m\rceil}
&\leq\left(\frac{em}{\lceil \beta m\rceil}\right)^{\lceil \beta m\rceil} \nonumber\\
&\leq\exp\left(2m\beta\log\frac{1}{\beta}\right).\label{eq:maj_binom}
\end{align}
(The first inequality is a classical result regarding binomial coefficients.)

We combine Equations \eqref{eq:P_Ax01S} and \eqref{eq:maj_binom}: Property \eqref{eq:Ax01S} is satisfied for any $S$ of cardinality $\lceil\beta m\rceil$ with probability at least
\begin{equation*}
1-\exp\left(-\frac{\beta^3}{e}m\right),
\end{equation*}
provided that $m$ is larger that some constant which depends on $\beta$.

If it is satisfied for any $S$ of cardinality $\lceil\beta m\rceil$, then it is satisfied for any $S$ of cardinality larger than $\beta m$, which implies the result.
\end{proof}

\subsubsection{Proof of Lemma \ref{lem:S_leq_bm}\label{sss:S_leq_bm}}

\begin{proof}[Proof of Lemma \ref{lem:S_leq_bm}]
We first assume $S$ to be fixed, of cardinality exactly $\lceil\beta m\rceil$.

Any vector $y\in\Range(A)$ is of the form $y=Av$, for some $v\in\C^n$. Inequality \eqref{eq:y1S} can then be rewritten as:
\begin{equation}\label{eq:S_leq_bm_rewritten}
||A_S v||=||\mathrm{Diag}(1_S)Av||\leq 10\sqrt{\beta\log(1/\beta)}||Av||,
\end{equation}
where $A_S$, by definition, is the submatrix obtained from $A$ by extracting the rows whose indexes are in $S$.

We apply Proposition \ref{prop:davidson} to $A$ and $A_S$, respectively for $t=\frac{1}{2}$ and $t=3\sqrt{\log(1/\beta)}$. It guarantees that the following properties hold:
\begin{gather*}
\inf_{v\in\C^n}\frac{||Av||}{||v||}\geq \sqrt{m}\left(\frac{1}{2}-\sqrt{\frac{n}{m}}\right);\\
\sup_{v\in\C^n}\frac{||A_Sv||}{||v||}\leq \sqrt{\Card S}\left(1+\sqrt\frac{n}{\Card S}+3\sqrt{\log(1/\beta)}\right),
\end{gather*}
with respective probabilities at least
\begin{gather*}
1-2\exp\left(-\frac{m}{4}\right);\\
\mbox{and }1-2\exp\left(-9(\Card S)\log(1/\beta)\right)\\
\hskip 2cm \geq 1-2\exp\left(-9\beta\log(1/\beta)m\right).
\end{gather*}
Assuming $m\geq Mn$ for some $M>0$, we deduce from these inequalities that, for all $v\in\C^n$,
\begin{align}
||A_Sv||&\leq \sqrt\frac{\Card S}{m}\left(\frac{1+\sqrt\frac{n}{\Card S}+3\sqrt{\log (1/\beta)}}{\frac{1}{2}-\sqrt{\frac{n}{m}}}\right)||Av||\nonumber\\
&\leq \sqrt{\beta+\frac{1}{m}}\left(\frac{1+\sqrt{\frac{1}{\beta M}}+3\sqrt{\log(1/\beta)}}{\frac{1}{2}-\sqrt{\frac{1}{M}}}\right)||Av||,\label{eq:tmp}
\end{align}
with probability at least
\begin{equation*}
1-2\exp\left(-9\beta\log(1/\beta)m\right)-2\exp\left(-\frac{m}{4}\right).
\end{equation*}
If we choose $M$ large enough, we can upper bound Equation \eqref{eq:tmp} by $(\epsilon+2\sqrt{\beta}(1+3\sqrt{\log(1/\beta)}))||Av||\leq (\epsilon+8\sqrt{\beta}\sqrt{\log(1/\beta)})$ for any fixed $\epsilon>0$. So this inequality implies Equation \eqref{eq:S_leq_bm_rewritten}.

As in the proof of Lemma \ref{lem:S_geq_bm}, there are at most
\begin{equation*}
\exp\left(2m\beta\log\frac{1}{\beta}\right)
\end{equation*}
subsets of $\{1,\dots,m\}$ with cardinality $\lceil\beta m\rceil$, as soon as $m$ is large enough. As a consequence, Equation \eqref{eq:S_leq_bm_rewritten} holds for any $v\in\C^n$ and $S$ of cardinality $\lceil\beta m\rceil$ with probability at least
\begin{equation*}
1-2\exp\left(-7\beta\log(1/\beta)m\right)-2\exp\left(-\left(\frac{1}{4}-2\beta\log\frac{1}{\beta}\right)m\right).
\end{equation*}
When $\beta\leq \frac{1}{100}$, we have
\begin{equation*}
\frac{1}{4}-2\beta\log\frac{1}{\beta}> 0,
\end{equation*}
so the resulting probability is larger than
\begin{equation*}
1-C_1\exp(-C_2 m),
\end{equation*}
for some well-chosen constants $C_1,C_2>0$.

This ends the proof. Indeed, if Equation \eqref{eq:S_leq_bm_rewritten} holds for any set of cardinality $\lceil\beta m\rceil$, it also holds for any set of cardinality $\Card S<\beta m$, because $||A_{S'} v||\leq ||A_S v||$ whenever $S'\subset S$. This implies Equation \eqref{eq:y1S}.
\end{proof}

\subsection{Proof of Lemma \ref{lem:second_term}\label{ss:second_term}}

\begin{lem*}[Lemma \ref{lem:second_term}]
For $M,C_1>0$ large enough, and $C_2>0$ small enough, when $m\geq M n$, the property
\begin{equation}\label{eq:second_term}
||\Im(v\odot\overline{\phase(Ax_0)})||\leq \frac{4}{5}||v||
\end{equation}
holds for any $v\in\Range(A)\cap \{Ax_0\}^\perp$, with probability at least
\begin{equation*}
1-C_1\exp(-C_2 m).
\end{equation*}
\end{lem*}

\begin{proof}
If we multiply $x_0$ by a positive real number, we can assume $||x_0||=1$. Moreover, as the law of $A$ is invariant under right multiplication by a unitary matrix, we can assume that
\begin{equation*}
x_0=\left(\begin{smallmatrix}1\\0\\\vdots\\0\end{smallmatrix}\right).
\end{equation*}
Then, if we write $A_1$ the first column of $A$, and $A_{2:n}$ the submatrix of $A$ obtained by removing this first column,
\begin{align}
&\Range(A)\cap\{Ax_0\}^\perp\nonumber\\
=&\left\{w-\frac{\scal{w}{A_1}}{||A_1||^2}A_1,
w\in\Range(A_{2:n})
\right\}.\label{eq:range_perp}
\end{align}
We first observe that
\begin{equation*}
\sup_{w\in\Range(A_{2:n})-\{0\}}\frac{|\scal{w}{A_1}|}{||w||}
\end{equation*}
is the norm of the orthogonal projection of $A_1$ onto $\Range(A_{2:n})$. The $(n-1)$-dimensional subspace $\Range(A_{2:n})$ has a rotationally invariant distribution in $\C^m$, and is independent of $A_1$. Thus, from Lemma \ref{lem:dasgupta} coming from \citep{dasgupta}, for any $t>1$,
\begin{equation*}
\sup_{w\in\Range(A_{2:n})-\{0\}}\frac{|\scal{w}{A_1}|}{||w||\,||A_1||}< \sqrt{\frac{t(n-1)}{m}},
\end{equation*}
with probability at least
\begin{equation*}
1-\exp\left(-(n-1)(t-1-\ln t)\right).
\end{equation*}
We take $t=\frac{m}{n-1}(0.04)^2$ (which is larger than $1$ when $m\geq Mn$ with $M>0$ large enough), and it implies that
\begin{equation}\label{eq:second_005}
\sup_{w\in\Range(A_{2:n})-\{0\}}\frac{|\scal{w}{A_1}|}{||w||\,||A_1||}< 0.04
\end{equation}
with probability at least
\begin{equation*}
1-\exp(-c_2m)
\end{equation*}
for some constant $c_2>0$, provided that $m\geq Mn$ with $M$ large enough.

Second, as $A_{2:n}$ is a random matrix of size $m\times(n-1)$, whose entries are independent and distributed according to the law $\mathcal{N}(0,1/2)+\mathcal{N}(0,1/2)i$, we deduce from Proposition \ref{prop:davidson} applied with $t=0.01$ that, with probability at least
\begin{equation*}
1-2\exp\left(-10^{-4}m\right),
\end{equation*}
we have, for any $x\in\C^{n-1}$,
\begin{align}
||A_{2:n}x|| &\geq \sqrt{m}\left(1-\sqrt{\frac{(n-1)}{m}}-0.01\right)||x||\nonumber\\
&\geq 0.98\sqrt{m}||x||,\label{eq:norm_C}
\end{align}
provided that $m\geq 10000n$.

We now set
\begin{equation*}
C = \mathrm{Diag}(\overline{\phase(A_1)})A_{2:n}.
\end{equation*}
The matrix $\left(\begin{matrix}\Im C&\Re C\end{matrix}\right)$ has size $m\times(2(n-1))$; its entries are independent and distributed according to the law $\mathcal{N}(0,1/2)$. So by \citep[Thm II.13]{davidson} (applied with $t=0.01$), with probability at least
\begin{equation*}
1-\exp(-5.10^{-5}m),
\end{equation*}
we have, for any $x\in\R^{2(n-1)}$,
\begin{align}
\left|\left|\left(\begin{matrix}\Im C&\Re C\end{matrix}\right)x
\right|\right|&\leq \sqrt{\frac{m}{2}}\left(1+\sqrt\frac{2(n-1)}{m}+0.01\right)||x||\nonumber\\
&\leq 1.02\sqrt\frac{m}{2}||x||,\label{eq:norm_ImReC}
\end{align}
provided that $m\geq 20000n$.

When Equations \eqref{eq:norm_C} and \eqref{eq:norm_ImReC} are simultaneously valid, any $w=A_{2:n}w'$ belonging to $\Range(A_{2:n})$ satisfies:
\begin{align}
\left|\left|\Im(w\odot\overline{\phase(Ax_0)})\right|\right|
&=\left|\left|\Im(Cw')\right|\right|\nonumber\\
&=\left|\left|\begin{pmatrix}\Im C&\Re C\end{pmatrix}\begin{pmatrix}
\Re w'\\\Im w'\end{pmatrix}
 \right|\right|\nonumber\\
&\leq 1.02\sqrt\frac{m}{2}\left|\left|\begin{pmatrix}
\Re w'\\\Im w'\end{pmatrix}
 \right|\right|\nonumber\\
&=1.02\sqrt\frac{m}{2}||w'||\nonumber\\
&\leq \frac{1.02}{0.98\sqrt{2}}||A_{2:n}w'||\nonumber\\
&= \frac{1.02}{0.98\sqrt{2}}||w||\nonumber\\
&\leq 0.75 ||w||.\label{eq:norms_combined}
\end{align}

We now conclude. Equations \eqref{eq:second_005}, \eqref{eq:norm_C} and \eqref{eq:norm_ImReC} hold simultaneously with probability at least
\begin{equation*}
1-C_1\exp(-C_2 m)
\end{equation*}
for any $C_1$ large enough and $C_2$ small enough, provided that $m\geq Mn$ with $M$ large enough. Let us show that, on this event, Equation \eqref{eq:second_term} also holds. Any $v\in\Range(A)\cap\{Ax_0\}^\perp$, from Equality \eqref{eq:range_perp}, can be written as
\begin{equation*}
v=w-\frac{\scal{w}{A_1}}{||A_1||^2}A_1,
\end{equation*}
for some $w\in\Range(A_{2:n})$. Using Equation \eqref{eq:second_005}, then Equation \eqref{eq:norms_combined}, we get:
\begin{align*}
\left|\left|\Im(v\odot\overline{\phase(Ax_0)})\right|\right|
&\leq\left|\left|\Im(w\odot\overline{\phase(Ax_0)})\right|\right|\\
&\hskip 1.5cm +\left|\left|\frac{\scal{w}{A_1}}{||A_1||^2}A_1\right|\right|\\
&\leq\left|\left|\Im(w\odot\overline{\phase(Ax_0)})\right|\right|\\
&\hskip 1.5cm +0.04 ||w||\\
&\leq 0.79 ||w||.
\end{align*}
But then, by Equation \eqref{eq:second_005} again,
\begin{equation*}
||v||^2=||w||^2-\frac{\scal{w}{A_1}^2}{||A_1||^2}\geq (1-(0.04)^2)||w||^2.
\end{equation*}
So
\begin{align*}
\left|\left|\Im(v\odot\overline{\phase(Ax_0)})\right|\right|
&\leq 0.79 ||w||\\
&\leq \frac{0.79}{\sqrt{1-(0.04)^2}}||v||\\
&\leq \frac{4}{5}||v||.
\end{align*}

\end{proof}

\bibliographystyle{plainnat}
\bibliography{../../bib_articles.bib,../../bib_proceedings.bib,../../bib_livres.bib,../../bib_misc.bib}

\end{document}